\documentclass[12pt,reqno,a4paper]{amsart}

 \usepackage{color}
\usepackage{amsmath,amssymb,amsfonts,amsthm,bm}
\usepackage{mathrsfs}
\usepackage{graphicx}
\usepackage{enumerate}

\usepackage[numbers, sort&compress]{natbib}
 \usepackage[shortlabels]{enumitem}
\usepackage[left=1.05in, right=1.05in]{geometry}
 \newtheorem{theorem}{Theorem}[section]
 \newtheorem{corollary}[theorem]{Corollary}
 \newtheorem{lemma}[theorem]{Lemma}
 
 \theoremstyle{assumption}
   
\theoremstyle{definition}

 \theoremstyle{remark}
\newtheorem{remark}[theorem]{Remark}
\numberwithin{equation}{section}

 \newcommand{\eps}{\varepsilon}
 
\newcommand{\norm}[1]{\Vert#1\Vert}
 \newcommand{\abs}[1]{\left\vert#1\right\vert}
\newcommand{\set}[1]{\left\{#1\right\}}
 \newcommand{\inner}[1]{\left(#1\right)}
\newcommand{\comi}[1]{\left<#1\right>}
\newcommand{\com}[1]{\left[#1\right]}
 \newcommand{\normm}[1]{{ \vert\kern-0.25ex \vert\kern-0.25ex \vert #1 
     \vert\kern-0.25ex \vert\kern-0.25ex \vert}}

\makeatletter

\@namedef{subjclassname@2020}{%
  \textup{2020} Mathematics Subject Classification}

\def\@startsection#1#2#3#4#5#6{%
 \if@noskipsec \leavevmode \fi
 \par \@tempskipa #4\relax
 \@afterindentfalse
 \ifdim \@tempskipa <\z@ \@tempskipa -\@tempskipa \@afterindentfalse\fi
 \if@nobreak \everypar{}\else
     \addpenalty\@secpenalty\addvspace\@tempskipa\fi
 \@ifstar{\@dblarg{\@sect{#1}{\@m}{#3}{#4}{#5}{#6}}}%
         {\@dblarg{\@sect{#1}{#2}{#3}{#4}{#5}{#6}}}%
}

\def\@settitle{%
  \bgroup
  \centering
  \vglue1cm
  \fontsize{12}{15}\fontseries{b}\selectfont
  \uppercasenonmath\@title
  \@title
  \vskip20pt plus 6pt minus 8pt
  \egroup
}

\def\@setauthors{%
  \begingroup
  \trivlist
  \centering \bfseries
 \normalsize\@topsep30\p@\relax
  \advance\@topsep by -\baselineskip
  \item\relax
  \andify\authors
 {\rmfamily\authors}%
  \endtrivlist
  \endgroup
}

\def\@setaddresses{\par
  \nobreak \begingroup
\normalsize
  \def\author##1{\nobreak\addvspace\bigskipamount}%
  \def\\{\unskip, \ignorespaces}%
  \interlinepenalty\@M
  \def\address##1##2{\begingroup
    \par\addvspace\bigskipamount\noindent
    \@ifnotempty{##1}{(\ignorespaces##1\unskip) }%
    {\ignorespaces##2}\par\endgroup}%
  \def\curraddr##1##2{\begingroup
    \@ifnotempty{##2}{\nobreak\indent{\itshape Current address}%
      \@ifnotempty{##1}{, \ignorespaces##1\unskip}\/:\space
      ##2\par}\endgroup}%
  \def\email##1##2{\begingroup
    \@ifnotempty{##2}{\nobreak\noindent{\itshape E-mail address}%
      \@ifnotempty{##1}{, \ignorespaces##1\unskip}\/: 
       ##2\par}\endgroup}%
   \def\urladdr##1##2{\begingroup
    \@ifnotempty{##2}{\nobreak\indent{\itshape URL}%
      \@ifnotempty{##1}{, \ignorespaces##1\unskip}\/:\space
      \ttfamily##2\par}\endgroup}%
  \addresses
  \endgroup
}

 \renewcommand\section{\@startsection{section}{1}{\z@}%
{27pt plus 6pt minus 8pt}{14pt plus 6pt minus 8pt}
{\center\normalfont\large\bfseries}}
\renewcommand\subsection{\@startsection{subsection}{2}{\z@}%
{27pt plus 6pt minus 8pt}{14pt plus 6pt minus 8pt}
{ \center
\normalfont\bfseries}}

\def\subsubsection{\@startsection{subsubsection}{3}%
  \z@{.5\linespacing\@plus.7\linespacing}{-.5em}%
  {\normalfont\itshape}}

\makeatother

\begin{document}

\title[Weighted  and  hypoelliptic    estimates  for the 
Fokker-Planck Operator]
{Weighted  and maximally hypoelliptic    estimates  for the 
Fokker-Planck Operator with electromagnetic fields}

\author[W.-X. Li]{Wei-Xi Li}

\address[W.-X. Li]{ School of Mathematics and Statistics, Wuhan University, Wuhan 430072, China \& Hubei Key Laboratory of Computational Science, Wuhan University, Wuhan 430072, China
  }

 \email{wei-xi.li@whu.edu.cn}

 \author[J. Zeng]{Juan Zeng}

\address[J.Zeng]{ School of Mathematics and Statistics, Wuhan University, Wuhan 430072, China  
  }

 \email{juan-zeng@whu.edu.cn}



\begin{abstract}
 We consider a Fokker-Planck operator with  electric potential and  electromagnetic fields.  We establish the sharp weighted and subelliptic estimates,  involving the control of the  derivatives of  electric potential and  electromagnetic fields.  Our proof relies on a  localization argument as well as 
 a careful calculation on commutators. 
\end{abstract}

\maketitle  

\section{Introduction and main results}

There have been several works on the Fokker-Planck operator with electric potential $V(x)$ which is
\begin{eqnarray}\label{FP++}
  K=y\cdot\partial_x-\partial_x V(x)\cdot
  \partial_y-\triangle_y+\frac{\abs y^2}{4}-\frac{n}{2},\quad
  (x,y)\in\mathbb{R}^{2n},
\end{eqnarray}
where $x$ denotes the space variable and $y$ denotes the velocity
variable, and $V(x)$ is a potential defined in the whole space
$\mathbb{R}_x^n$.  It is a degenerate operator with the absence of  diffusion in $x$ variable, and  can be seen as a    Kolmogorov-type  operator.   The classical hypoelliptic techniques and their global counterparts have been developed recently to establish  global  estimates and to investigate the short and long time behavior and the spectral properties for Fokker-Plack operator K in \eqref{FP++}. We refer to Helffer-Nier's notes \cite{MR2130405} for the comprehensive argument on this topic, seeing also the earlier work \cite{MR2034753}  of H\'erau-Nier.   In the first author's work \cite{MR3783786, MR3060700} we improved the previous result  and gave a new criterion involving the microlocal property of potential $V$. Here   we also mention the very recent progress made by  Ben Said-Nier-Viola \cite{MR4156527}  and Ben Said \cite{2018arXiv181206645B}.  Finally as a result of the global   estimates   it enables to answer partially a conjecture stated by Helffer-Nier \cite{MR2130405} which says Fokker-Planck operator $K$ has a compact resolvent if  and only if Witten Laplacian has a compact resolvent.   The necessity part is well-known and the  reverse
implication still remains open with some partial answers; in fact various hypoelliptic techniques,  such as Kohn' method and nilpotent approach (e.g.\cite{MR897103, MR660652, MR436223}),  were developed to establish the resolvent criteria for these two different type operators (see  \cite{MR2130405, MR3775156, MR3783786, MR3060700}).      

Inspired by the recent work  of Helffer-Karaki \cite{2020arXiv201016175H}, we consider here a more general Fokker-Planck operator with  electromagnetic fields besides the electric potential, which reads
\begin{eqnarray}\label{FPE}
  P=y\cdot\partial_x-\partial_x V(x)\cdot
  \partial_y- H(x)\cdot \inner{y\wedge \partial_y} -\triangle_y+\frac{\abs y^2}{4}-\frac{n}{2},\quad
  (x,y)\in\mathbb{R}^{2n},
\end{eqnarray}
where $n=2$ or $3$ and  $H(x)$ is a scalar function of $x$  for $n=2$ and  a vector field $\inner{H_1(x), H_2(x), H_3(x)}$ of only $x$-variable  for $n=3$,  and  $y\wedge \partial_y $ is defined by
\begin{eqnarray*}
	 y\wedge \partial_y=\left\{
	 \begin{aligned}
	 	&y_1\partial_{y_2}-y_2\partial_{y_1},\quad n=2,\\
	 &	\big(y_2\partial_{y_3}-y_3\partial_{y_2},\   y_3\partial_{y_1}-y_1\partial_{y_3}, \ y_1\partial_{y_2}-y_2\partial_{y_1}\big),\quad n=3.
	 \end{aligned}
	 \right.
\end{eqnarray*}
The operator is initiated  by Helffer-Karaki \cite{2020arXiv201016175H},  where they established the maximal estimate by virtue of nilpotent approach,   giving a criteria for the compactness of the resolvent.  Here we aim  to give another proof, basing on a localization argument and a careful calculation on commutators.  
 Note the operator $P$ in \eqref{FPE} is reduced to the operator $K$ given  \eqref{FP++} for $H\equiv 0$; meanwhile the maximal estimates for the Fokker-Planck operator with pure electromagnetic fields (i.e.,$V\equiv 0$) was investigated by Zeinab Karaki \cite{2019arXiv190108120K}.

Before stating our main result we first introduce some notations used throughout the paper. We will use $\Vert \cdot\Vert_{L^2}$ to denote
the norm of the complex Hilbert space
$L^2\big(\mathbb{R}^{2n}\big),$ and denote by $
C_0^\infty\inner{\mathbb{R}^{2n}}$ the set of smooth compactly
supported functions.    Denote  by  $\mathscr F_x$ the (partial) Fourier transform with respect to $x$ and by $\xi$ the  Fourier dual variable of $x$.  Throughout the paper we use the notation $\comi{\cdot}=\inner{1+\abs{\cdot}^2}^{1/2}$ and let $\comi{D_x}^r=\inner{1-\Delta_x}^{r/2}$  be the Fourier multiplier with symbol $\comi{\xi}^r$,  that is, 
$$\forall\ u\in C_0^\infty(\mathbb R^{2n}),\quad \mathscr F_x \inner{\comi{D_x}^r u}(\xi)=\comi{\xi}^r \mathscr F_x  u(\xi).$$
Similarly we can define $\comi{D_y}$.

\begin{theorem}\label{+Hypo}
Let $V(x) \in C^2(\mathbb{R}^n)$ with $n=2$ or $3$  be  a real-valued function and let $H(x)$ be a continuous    real vector-valued function. Suppose there exists a constant $C_0$ such   
that for any $x\in\mathbb R^n$ we have
\begin{equation}
	\label{eh}
	\abs{H(x)}\leq C_0  \comi{ \partial_x V(x)  }^{\delta}\  {\rm with}\  \delta<2/3,
\end{equation}
and
\begin{equation}\label{+HyP}
   \forall~\abs\alpha=2, \quad \abs{\partial_x ^\alpha V(x)}\leq C_0  \comi{\partial_x V(x) }^{s}
    ~~\,\,{\rm with~~s<\frac{4}{3}}.
\end{equation}
Then we can find a constant $C,$ depending on the above $C_0$ and $s$,  such that 
  \begin{equation}\label{weg}
\forall\ u\in
 C_0^\infty\big(\mathbb{R}^{2n}\big),\quad \norm{ \comi{\partial_x  V}^{2/3}u}_{L^2}    
   \leq C\Big\{\,
   \norm{P u}_{L^2}+\norm{u}_{L^2}\Big\}.
\end{equation}
Moreover if $H$ satisfies additionally that  $H \in C^1(\mathbb R^{n})$ and 
\begin{equation}\label{ead}
	\abs{\partial_xH(x)}\leq C_0  \comi{ \partial_x V(x)  }^{s/2} 
\end{equation}
with $s$ given in \eqref{+HyP}, 
then we have following subelliptic estimate 
  \begin{equation}\label{+++A1}
\forall\ u\in
 C_0^\infty\big(\mathbb{R}^{2n}\big),\quad     \norm{\comi{D_x}^{\tau}u}_{L^2}
   \leq C\Big\{\,
   \norm{P u}_{L^2}+\norm{u}_{L^2}\Big\},
\end{equation}
where
\begin{eqnarray*}
	\tau=\left\{
	\begin{aligned}
		&2/3, \quad {\rm if}\  s\leq 2/3,\\
		&(4-3s)/3, \quad {\rm if}\  2/3<s\leq 10/9,\\
		&(4-3s)/6, \quad {\rm if}\   10/9< s< 4/3.
	\end{aligned}
\right.
\end{eqnarray*}
\end{theorem}

Note if the number $s$ in \eqref{+HyP} is less than or equal to $2/3$ then we obtain the sharp subelliptic exponent $\tau=2/3$.  This enables to obtain   the  maximal estimate stated as below (see Section \ref{sec22}).

\begin{corollary}\label{cormax}
	If  $V$ satisfies \eqref{+HyP} with $s\leq 2/3$, and $H$ satisfies the conditions \eqref{eh} and \eqref{ead} . Then we have the following maximal estimate, for any $u\in
 C_0^\infty\big(\mathbb{R}^{2n}\big)$,
	\begin{multline*}
		\norm{\comi{\partial_x  V}^{2/3}u}_{L^2}+\norm{\comi{D_x}^{2/3}u}_{L^2}+\norm{\big(y\cdot\partial_x-\partial_x V(x)\cdot
  \partial_y- H(x)\cdot \inner{y\wedge \partial_y}\big)u}_{L^2} \\
  +\norm{\comi{D_y}^2u}_{L^2}+\norm{\comi{y}^2u}_{L^2}\leq 	 C\Big\{\,
   \norm{P u}_{L^2}+\norm{u}_{L^2}\Big\}.
\end{multline*}
\end{corollary}

\begin{remark}  The above maximal estimate was established  by Helffer-Karaki \cite{2020arXiv201016175H} under  similar  conditions as that in \eqref{eh},\eqref{+HyP}  and \eqref{ead}, but therein they require $\delta=0$ and $s<2/3$. 
  The result in Theorem \ref{+Hypo} generalizes  the one  established by the first author  \cite{MR3060700},  considered therein is a specific case  of $H\equiv0$. 
\end{remark}

  Another consequence of Theorem \ref{+Hypo} is to  
 analyze the compact criteria for resolvent of Fokker-Planck operator $P$ in \eqref{FPE}.  Due to the weighted estimate \eqref{+HyP} we see  the Fokker-Planck operator $P$ admits a compact resolve  if  $\abs{\partial_xV(x)}\rightarrow+\infty$ as $\abs{x}\rightarrow +\infty$. Moreover  as in the purely electric case (i.e., $H\equiv 0$),  $P$ is closed linked with Witten Laplace operator $\triangle_{V/2}^{(0)}$ defined by
\[
  \triangle_{V/2}^{(0)}=
  -\triangle_x+\frac{1}{4}\abs{\partial_xV(x)}^2-\frac{1}{2}\triangle_x
  V(x).
\]
In fact we can repeat the argument for proving \cite[Corollary 1.3]{MR3060700} to conclude the following 
 
 \begin{corollary} 
  Let $H(x)$ and $V(x)$ satisfy the conditions \eqref{eh}, \eqref{+HyP} and  \eqref{ead}.
  Then the Fokker-Planck operator $P$   in \eqref{FPE} has a compact resolvent if the Witten
  Laplacian $\triangle_{V/2}^{(0)}$ has a compact resolvent.
\end{corollary}

The paper is organized as follow. In Sections \ref{sec3} and \ref{sec4} we  prove, respectively, the weighted estimate and the subelliptic estimate in Theorem \ref{+Hypo}.  The last section is devoted to proving Corollary \ref{cormax}, the maximal estimate.

\section{Weighted estimate}\label{sec3}

In this part we prove the   weighted estimate    \eqref{weg} in Theorem \ref{+Hypo}.  From now on we use the notation $Q=y \cdot \partial_x-\partial_x V(x)\cdot
 \partial_y-H(x)\cdot\inner{y\wedge \partial_y }$ and $L_j=\partial_{y_j}+\frac{y_j}{2},j=1,\cdots n.$  Then we can
 rewrite the Fokker-Planck operator $P$  in \eqref{FPE}  as
\begin{equation}\label{+FP+}
  P=Q+\sum_{j=1}^nL_j^*L_j.
\end{equation}
By direct verification the following estimates
\begin{equation}\label{elv}
 \sum_{1\leq j\leq n} \norm{L_j u}_{L^2}^2\leq  {\rm Re} \comi{P u,\ u}_{L^2}   
\end{equation}
and
\begin{equation}\label{uy}
 \norm{\comi{y}u}_{L^2}+\norm{\comi{D_y}u}_{L^2}  \leq C\sum_{1\leq j\leq n}\norm{L_j u}_{L^2}+C\norm{ u}_{L^2}
\end{equation}
	hold for any $u\in C_0^\infty\big(\mathbb{R}^{2n}\big)$.
 Here and below to simplify the notation  we will use   the
capital letter $C$ to denote different suitable constants.

\begin{lemma}\label{lewe}
Let $H$ and $V$ satisfy \eqref{eh} and \eqref{+HyP} respectively. Then
for any $\eps>0$ we can find a constant $C_\eps$ such that
the following estimate
	\begin{equation*} 
 \sum_{j=1}^n \inner{\norm{L_j\comi{\partial_xV }^{{1/3}}
  u}_{L^2}^2+\norm{L_j\comi{y} 
  u}_{L^2}^2}\leq \eps \norm{\comi{\partial_xV }^{{2/3}}
  u}_{L^2}^2+C_\eps\set{\norm{Pu}_{L^2}^2+\norm{u}_{L^2}^2}
\end{equation*}
holds for all $u\in C_0^\infty\big(\mathbb{R}^{2n}\big)$.
\end{lemma}

\begin{proof}[Proof of this lemma]
	we use \eqref{elv} to get 
\begin{multline*}
 \sum_{1\leq j\leq n} \norm{L_{j}\comi{\partial_xV(x)}^{1/3}u}_{L^2}^2
   \leq {\rm Re}\comi{P\comi{\partial_xV(x)}^{1/3}u,~\comi{\partial_xV(x)}^{1/3}u}_{L^2}\\
   \leq \eps\norm{\comi{\partial_xV(x)}^{2/3}u}_{L^2}^2+C_\eps  \norm{Pu}_{L^2}^2+
  {\rm Re}\comi{\big[P,~\comi{\partial_xV(x)}^{1/3}\big]u,~\comi{\partial_xV(x)}^{1/3}u}_{L^2}.
\end{multline*}
Moreover using \eqref{+HyP} yields, for any $\eps_1, \eps>0,$
\begin{eqnarray*}
	&&{\rm Re}\comi{\big[P,~\comi{\partial_xV(x)}^{1/3}\big]u,~\comi{\partial_xV(x)}^{1/3}u}_{L^2}\\
	&& \leq
 \eps_1\norm{\comi{y}\comi{\partial_xV(x)}^{1/3}u}_{L^2}^2+C_{\eps_1}\norm{\comi{\partial_xV(x)}^{s-{2/3}}u}_{L^2}^2\\
 && \leq \frac{1}{2} \sum_{j=1}^n
\norm{L_j \comi{\partial_xV(x)}^{1/3}u}_{L^2}^2+\eps \norm{\comi{\partial_xV(x)}^{{2/3}}u}_{L^2}^2+C_{\eps}\norm{u}_{L^2}^2,
\end{eqnarray*}
the last inequality using \eqref{uy} and the fact that $s<4/3.$ Combining the above estimates we obtain
\begin{equation}\label{fre}
	\sum_{j=1}^n  \norm{L_j\comi{\partial_xV(x)}^{{1/3}}
  u}_{L^2}^2 \leq \eps \norm{\comi{\partial_xV(x)}^{{2/3}}
  u}_{L^2}^2+C_\eps\big\{\norm{Pu}_{L^2}^2+\norm{u}_{L^2}^2\big\}
\end{equation}
Similarly, using again \eqref{elv} and \eqref{uy},
\begin{eqnarray*}
\begin{aligned}
 & \sum_{1\leq j\leq n} \norm{L_j\comi{y}u}_{L^2}^2
  \leq    {\rm Re}\comi{P\comi{y}u,~\comi{y}u}_{L^2}   \leq  \abs{\comi{\com{P,~\comi{y}}u,~\comi{y}u}_{L^2}}+ \norm{Pu}_{L^2} \norm{\comi{y}^{2}u}_{L^2}\\
  &\leq \eps_1 \sum_{1\leq j\leq n} \norm{L_j\comi{y}u}_{L^2}^2+C_{\eps_1}\Big\{\norm{Pu}_{L^2}^2+\norm{u}_{L^2}^2\Big\} +\abs{\comi{\com{P,~\comi{y}}u,~\comi{y}u}_{L^2}},
  \end{aligned}
\end{eqnarray*}
with $\eps_1>0$ arbitrarily small.  Moreover it follows from the assumption \eqref{eh} that,  for any $\eps,\eps_1>0,$
\begin{eqnarray*}
\begin{aligned}
& \abs{\comi{\com{P,~\comi{y}}u,~\comi{y}u}_{L^2}}  \leq   C\norm{\comi{\partial_xV }^{2/3}u}_{L^2} \norm{\comi{y}\comi{\partial_xV }^{1/3}u}_{L^2}\\
  &\qquad\qquad+C
  \norm{\comi{y}^2 u}_{L^2} \norm{\comi{\partial_xV }^{\delta}u}_{L^2}+C
  \norm{\comi{y} u}_{L^2} \norm{\comi{D_y}u}_{L^2}   \\
 & \leq \eps \norm{\comi{\partial_xV }^{2/3}u}_{L^2}^2+\eps_1 \sum_{1\leq j\leq n} \norm{L_j\comi{y}u}_{L^2}^2\\
 &\quad+C_{\eps,\eps_1}\set{\sum_{1\leq j\leq n}\norm{L_j\comi{\partial_xV }^{1/3}u}_{L^2}^2+
  \norm{Pu}_{L^2}^2+\norm{u}_{L^2}^2},
  \end{aligned}
\end{eqnarray*}
where the last inequality holds because of \eqref{uy} and the fact that $\delta<2/3$. As a result combining the above estimates and choosing $\eps_1$ small enough, we conclude
\begin{eqnarray*}
\begin{aligned}
	\sum_{1\leq j\leq n} \norm{L_j\comi{y}u}_{L^2}^2&\leq \eps \norm{\comi{\partial_xV }^{2\over3}u}_{L^2}^2 +C_{\eps }\set{\sum_{1\leq j\leq n}\norm{L_j\comi{\partial_xV }^{1\over3}u}_{L^2}^2+
  \norm{Pu}_{L^2}^2+\norm{u}_{L^2}^2}\\
&  \leq \eps \norm{\comi{\partial_xV }^{2\over3}u}_{L^2}^2 +C_{\eps }\set{ 
  \norm{Pu}_{L^2}^2+\norm{u}_{L^2}^2},
  \end{aligned}
\end{eqnarray*}
the last inequality using \eqref{fre}.  This with \eqref{fre} completes the proof  of Lemma \ref{lewe}.
\end{proof}

\begin{proof}[Proof of Theorem \ref{+Hypo}: weighted estimate]  Here we will prove the weighted estimate \eqref{weg} in Theorem \ref{+Hypo}.   Let $M\in
C^1\big(\mathbb{R}^{2n}\big)$ be a real-valued function given by
\[
  M=M(x,y)= 2\comi{\partial_xV(x)}^{-{2/3}}\partial_x V(x)\cdot y.
\]
We use the fact that   $ \abs {M(x,y)} \leq
  C \abs{y}\comi{\partial_xV(x)}^{1/3}$ and
\[
  {\rm Re} \comi{Pu,~Mu}_{L^2}={\rm Re}
  \comi{Qu,~Mu}_{L^2}+
  {\rm Re} \sum_{j=1}^n\comi{L_j^*L_ju,~Mu}_{L^2}
\] 
due to \eqref{+FP+}, to conclude,  by virtue of \eqref{uy},
\begin{equation}\label{rmain}
\begin{aligned}
  &{\rm Re}
  \comi{Qu,~Mu}_{L^2}\leq
  \norm{Pu}_{L^2}^2+C\norm{ \comi{y}\comi{\partial_xV}^{{1/3}} u}_{L^2}^2+\sum_{j=1}^n\abs{\comi{L_j^*L_ju,~Mu}_{L^2}}\\
  &\leq  \norm{Pu}_{L^2}^2  +C\sum_{1\leq j\leq n} \inner{\norm{L_j\comi{\partial_xV }^{ {1\over3}} u}_{L^2}+\norm{L_j\comi{y } u}_{L^2}+\norm{\comi{\partial_xV}^{1\over 3}u}_{L^2}^2}\\
  &\leq  \eps\norm{\comi{\partial_xV}^{2/3}u}_{L^2}^2+C_\eps\inner{ \norm{Pu}_{L^2}^2 + \norm{u}_{L^2}^2 },
  \end{aligned}
\end{equation}
where  we use Lemma \ref{lewe} in the last line.
As for the term on left side we 
use the fact that 
${\rm Re}\comi{Q u,~Mu}_{L^2}= \frac{1}{2}\comi{[M,\ Q] u,~u}_{L^2}$
and 
\begin{eqnarray*}
	\frac{1}{2}[M,\ Q] =-y\cdot\partial_x\big(\comi{\partial_xV }^{-{2/3}}\partial_x V \cdot
  y\big)+ \comi{\partial_xV }^{-{2/3}}\abs{\partial_x
  V }^2 
  +\comi{\partial_xV }^{-{2/3}}H \cdot\inner{y\wedge \partial_xV }
\end{eqnarray*}
to compute, using \eqref{eh} and \eqref{+HyP} as well as Lemma \ref{lewe}, 
 \begin{eqnarray*}
  \norm{\comi{\partial_xV }^{{2/3}}u}_{L^2}^2&\leq& {\rm Re}\comi{Q u, Mu}_{L^2}+ 
  \norm{u}_{L^2}^2+\eps   \norm{\comi{\partial_xV }^{{2/3}}u}_{L^2}^2+C_\eps \norm{\comi{y} \comi{\partial_xV }^{ {1/3}} u}_{L^2}\\
  &\leq& {\rm Re}\comi{Q u, Mu}_{L^2}+ \eps \norm{\comi{\partial_xV(x)}^{{2/3}}
  u}_{L^2}^2+C_\eps\big\{\norm{Pu}_{L^2}^2+\norm{u}_{L^2}^2\big\},
\end{eqnarray*}
and thus,  letting $\eps=1/2$ above,
\begin{eqnarray*}
  \norm{\comi{\partial_xV }^{{2/3}}u}_{L^2}^2\leq  2 {\rm Re}\comi{Q u, Mu}_{L^2}+ C\big\{\norm{Pu}_{L^2}^2+\norm{u}_{L^2}^2\big\}.
\end{eqnarray*}
As a result, we combine   the above estimate with  \eqref{rmain} to get
\begin{eqnarray*}
  \norm{\comi{\partial_xV }^{{2/3}}u}_{L^2}^2\leq \eps\norm{\comi{\partial_xV}^{2/3}u}_{L^2}^2+C_\eps\set{ \norm{Pu}_{L^2}^2 + \norm{u}_{L^2}^2 },
\end{eqnarray*}
which  gives the desired weighted  estimate \eqref{weg} if we let $\eps$ be small enough. 
\end{proof}

 As an immediate consequence of Lemma \ref{lewe} and the weighted estimate \eqref{weg} we see the estimate 
 \begin{equation} \label{lje} 
 \sum_{j=1}^n \inner{\norm{L_j\comi{\partial_xV }^{{1/3}}
  u}_{L^2}^2+\norm{L_j\comi{y} 
  u}_{L^2}^2}\leq C \set{\norm{Pu}_{L^2}^2+\norm{u}_{L^2}^2}
\end{equation}
holds for all $u\in C_0^\infty\big(\mathbb{R}^{2n}\big)$, provided 
   $H$ and $V$ satisfy \eqref{eh} and \eqref{+HyP} respectively.

  \section{Subelliptic estimate}
\label{sec4}

In this section we will prove the subelliptic estimate   \eqref{+++A1} in Theorem \ref{+Hypo}.   The proof relies on
a localization argument. 
Firstly we recall some standard results concerning the partition
of unity.  For more detail we refer to \cite{Hormander85, MR2599384} for
instance. Let $g$ be a metric of the following form
\begin{equation}\label{7292}
  g_x=\comi{\partial_xV(x)}^{s}\abs{dx}^2,\quad x\in\mathbb{R}^n,
\end{equation}
where $s$ is  the real number given in \eqref{+HyP}.

\begin{lemma}[Lemma 4.2 in \cite{MR3060700}]\label{lem729}
  Suppose $V$ satisfies the condition \eqref{+HyP}. Then the metric
  $g$ defined by \eqref{7292} is slowly varying,  i.e., we can find two constants $C_*,r>0$ such
  that if $g_x(x-\tilde x)\leq r^2$ then $$C_*^{-1}\leq \frac{g_x}{g_{\tilde x}}\leq
  C_*.$$
\end{lemma}

\begin{lemma}[(Lemma 18.4.4. in \cite{Hormander85})]\label{uni}
  Let $g$ be a slowly varying metric. We can find a constant $r_0>0$ and a sequence $x_\mu\in\mathbb{R}^n,\mu\geq1,$ such
  that the union of the balls
  \[
    \Omega_{\mu,r_0}=\set{x\in\mathbb{R}^{n};\quad g_{x_\mu}\inner{x-x_\mu}<r_0^2}
  \]
  coves the whole space $\mathbb{R}^{n}.$ Moreover there exists a positive integer
  $N ,$ depending only on $r_0,$ such that the intersection
  of more than $N $ balls is always empty.
  One can choose a
  family of nonnegative functions $\set{\varphi_\mu}_{\mu\geq 1}$  such that
 \begin{equation}\label{73966}
   {\rm supp}~ \varphi_\mu\subset\Omega_{\mu,r_0},\quad
   \sum_{\mu\geq1} \varphi_\mu^2 =1~~\,\,{\rm and }~~\,\,
   \sup_{\mu\geq1}\abs{\partial_x \varphi_\mu (x)}\leq C \comi{\partial_xV(x)}^{s\over2}.
\end{equation}
\end{lemma}

 By  Lemmas \ref{lem729} and \ref{uni}  we can find a constant $C,$ such that for any
  $\mu\geq1$ one has
\begin{equation}\label{7310}
  \forall~x, \tilde x\in {\rm
  supp}~  \varphi_\mu,\quad C^{-1}\comi{\partial_xV(\tilde x)} \leq \comi{\partial_xV(x)} \leq C \comi{\partial_xV(\tilde x)}.
\end{equation}

  \begin{lemma}[Lemma 4.6 in \cite{MR3060700}]\label{lemreg}
     Let $\set{\varphi_\mu}_{\mu\geq1}$ be the partition given in Lemma \ref{uni},
     and let $a\in]0, 1/2[$ be a real number.
     Then there exists a constant $C,$ depending on the integer $N $
     given in Lemma \ref{uni}, such that for any $u\in
     C_0^\infty\big(\mathbb{R}^{2n}\big)$ we have
     \begin{equation*}
       \norm{\inner{1-\triangle_x}^{a}u}_{L^2}^2\leq
       C\sum_{\mu\geq1}\norm{\inner{1-\triangle_x}^{a}\varphi_\mu  u}_{L^2}^2
       +C\norm{Pu}_{L^2}^2+C\norm{u}_{L^2}^2.
     \end{equation*}
   \end{lemma}

  Let
$\set{\varphi_\mu}_{\mu\geq1}$ be the partition of unity given in
Lemma \ref{uni}.  For each  $x_\mu\in\mathbb{R}^{n}$  we define  the operator
  \[
    P_{x_\mu}=y\cdot\partial_x-\partial_x V(x_\mu)\cdot
    \partial_y-H(x_\mu)\cdot\inner{y\wedge\partial_y }  -\triangle_y+\frac{\abs y^2}{4}-\frac{n}{2}.
  \]
  Then   
  \begin{equation}\label{wh}
     \varphi_\mu P u=P_{x_\mu}\varphi_\mu u+ R_\mu u
  \end{equation}
  with 
  \begin{equation}\label{+Rem}
  R_\mu=-y\cdot\partial_x\varphi_\mu(x)
  -\varphi_\mu\inner{\partial_x V(x)-\partial_x V(x_\mu)}\cdot
  \partial_y-\varphi_\mu\inner{H(x)-H (x_\mu)}\cdot\inner{y\wedge 
  \partial_y}.
\end{equation}

  \begin{lemma}\label{patch}
Suppose $H(x)$ and $V(x)$ satisfy the conditions   \eqref{eh}-\eqref{+HyP} and \eqref{ead}.  Let $R_\mu$ be
the operator given in \eqref{+Rem}. Then
  \begin{equation}\label{CMU}
    \forall~u\in   C_0^\infty\big(\mathbb{R}^{2n}\big),\quad\sum\limits_{\mu\geq1}\norm{
    R_\mu
    u}_{L^2}^2 \leq C
    \set{\norm{P \comi{\partial_xV(x)}^{\tilde s}
    u}_{L^2}^2+\norm{P
    u}_{L^2}^2+\norm{u}_{L^2}^2},
  \end{equation}
  where $\tilde s=\frac{2}{3}-\tau$ with $\tau$ given in \eqref{+++A1}, i.e.,
  $\tilde s$ equals to $0$ if $s\leq {2/3}$,  $s-{2/3}$ if $ 2/3 < s\leq 10/9,$ and
   $ s/2$ if $ 10/9 < s< 4/3$.
\end{lemma}

\begin{proof} 
We write
\begin{equation}\label{rmu}
\sum\limits_{\mu\geq1}\norm{
    R_\mu
    u}_{L^2}^2 \leq I_1+I_2+I_3
\end{equation}
with
\begin{eqnarray*}
\begin{aligned}
I_1&=  2\sum\limits_{\mu\geq1}\norm{\inner{y\cdot\partial_x\varphi_\mu}
    u}_{L^2}^2,\\
	I_2&=  2\sum\limits_{\mu\geq1}\norm{\varphi_\mu \inner{\partial_x V(x)-\partial_x
    V(x_\mu)}\cdot \partial_y
    u}_{L^2}^2,\\
    I_3&=  2\sum\limits_{\mu\geq1}\norm{\varphi_\mu\inner{H(x)-H (x_\mu)}\cdot\inner{y\wedge 
  \partial_y}
    u}_{L^2}^2.
	\end{aligned}
\end{eqnarray*}
Note it is just finite sum of at most $N$ terms for each $I_k, 1\leq k\leq 3,$ recalling $N$ is the integer given in Lemma \ref{uni}.   It follows from the last inequality in \eqref{73966} that   
 \begin{eqnarray*}
	 I_1\leq  C  \norm{\comi{y}   \comi{\partial_xV(x)}^{s/2}u}_{L^2}^2.
	 \end{eqnarray*}
Similarly  observing $\abs{x-x_\mu}\leq C\comi{\partial_xV(x_\mu)}^{-{s\over2}}$ for any $x\in$
 supp~$\varphi_\mu,$
 we use the conditions \eqref{+HyP} and \eqref{ead} as well as  \eqref{7310}
    to compute
   \begin{eqnarray*}
   I_2 
     \leq C \norm{\comi{D_{y}}  \comi{\partial_xV(x)}^{s/2}u}_{L^2}^2 
   \end{eqnarray*}
and
\begin{eqnarray*}
	I_3\leq C\norm{\comi{D_y}\comi yu}_{L^2}^2\leq C\sum_{1\leq j\leq n}\norm{L_j\comi yu}_{L^2}^2+C\norm{\comi{y} u}_{L^2}^2\leq C \set{\norm{P
    u}_{L^2}^2+\norm{u}_{L^2}^2},
\end{eqnarray*}
the last inequality using \eqref{lje} as well as \eqref{elv} and \eqref{uy}.   As a result plugging the estimates on $I_k$ into \eqref{rmu}    yields
\begin{equation}\label{let}
	\sum\limits_{\mu\geq1}\norm{
    R_\mu
    u}_{L^2}^2 
    \leq        C\big( \norm{\comi{y}   \comi{\partial_xV }^{s\over2}u}_{L^2}^2+ \norm{\comi{D_y}   \comi{\partial_xV }^{s\over2}u}_{L^2}^2\big)+C \inner{\norm{P
    u}_{L^2}^2+\norm{u}_{L^2}^2}.
    \end{equation}
  It remains to control the first  two terms on the right side, and here we follow the argument in \cite{MR3060700} with modification.    

\noindent\underline{\it (a) The case of   $s\leq 2/3.$}
In such a case  we have
  \begin{multline*}
   \norm{\comi{y}   \comi{\partial_xV(x)}^{s\over2}u}_{L^2}^2+ \norm{\comi{D_y}   \comi{\partial_xV(x)}^{s\over2}u}_{L^2}^2\\
 \leq  \sum_{1\leq j\leq n}\norm{L_j  \comi{\partial_xV(x)}^{1/3}u}_{L^2}^2+C\norm{  \comi{\partial_xV(x)}^{1/3}u}_{L^2}^2 
   \leq C\set{
    \norm{Pu}_{L^2}^2+\norm{ u}_{L^2}^2},
  \end{multline*}
  the last inequality using the estimates \eqref{weg}  and \eqref{lje} that were established in the previous section.  This  with \eqref{let}  yields the validity of
    \eqref{CMU}
    for $s\leq {2/3}$.

\medskip
\noindent\underline{\it (b) The case of   $10/9<s<4/3.$}  We use 
   \eqref{elv} and \eqref{uy} to conclude, for any $u\in C_0^\infty\inner{\mathbb{R}^{2n}}$,
   \begin{eqnarray*}
   \begin{aligned}
 \norm{\comi{y}   \comi{\partial_xV }^{s\over2}u}_{L^2}^2+ \norm{\comi{D_y}   \comi{\partial_xV }^{s\over2}u}_{L^2}^2  
   &\leq C \norm{
    P \comi{\partial_xV }^{s\over2}u}_{L^2}^2+C \norm{ \comi{\partial_xV}^{s\over2}u}_{L^2}^2\\
   & \leq C\norm{
    P \comi{\partial_xV }^{s\over2}u}_{L^2}^2+C\norm{ Pu}_{L^2}
    +C\norm{u}_{L^2}^2,
    \end{aligned}
  \end{eqnarray*}
   the last inequality using again   the weighted estimate \eqref{weg} since $s<4/3$.
  This gives the validity of \eqref{CMU} for $10/9<s< 4/3.$

\medskip
\noindent\underline{\it (c) The case of   $2/3<s\leq 10/9.$}  In this case we use \eqref{elv} and \eqref{uy}   to compute 
  \begin{eqnarray*}
    && \norm{\comi{y}   \comi{\partial_xV }^{s\over2}u}_{L^2}^2+ \norm{\comi{D_y}   \comi{\partial_xV }^{s\over2}u}_{L^2}^2  \\ &&\leq   \abs{\comi{
    P \comi{\partial_xV }^{s\over2}u,~ \comi{\partial_xV }^{s\over2}u}_{L^2}}+C\norm{\comi{\partial_xV }^{s\over2}u}_{L^2}^2 \\
    &&\leq 
   C \abs{\comi{
    P \comi{\partial_xV }^{s-{2/3}}u,~ \comi{\partial_xV }^{{2/3}}u}_{L^2}}
    + C\norm{\comi{\partial_xV }^{s\over2}u}_{L^2}^2\\
    &&\quad+ C\abs{\comi{\com{
    P,~ \comi{\partial_xV }^{\frac{s}{2}-{2\over3}}} \comi{\partial_xV }^{{s\over2}}u,
    ~ \comi{\partial_xV }^{{2/3}}u}_{L^2}}\\
    &&\leq 
   C \norm{
    P \comi{\partial_xV }^{s-{2/3}}u}_{L^2}^2+ C_\eps\norm{\comi{\partial_xV }^{{2/3}}u}_{L^2}^2 
    + \eps\norm{\comi{y}\comi{\partial_xV }^{2s-5/3}u}_{L^2}^2.
  \end{eqnarray*}
Using the fact that $2s-{5\over3}\leq \frac{s}{2}$ for $s\leq 10/9,$ and letting $\eps$ above small enough, we get, in view of the weighted estimate \eqref{weg},
    \begin{eqnarray*}
\norm{\comi{y}   \comi{\partial_xV }^{s\over2}u}_{L^2}^2+ \norm{\comi{D_y}   \comi{\partial_xV }^{s\over2}u}_{L^2}^2
    \leq 
    C \set{\norm{
    P \comi{\partial_xV(x)}^{s-{2\over3}}u}_{L^2}^2+\norm{
    Pu}_{L^2}^2+\norm{u}_{L^2}^2}.
  \end{eqnarray*}
   Inserting the above inequality into \eqref{let}  
   we get the desired estimate \eqref{CMU} for
  ${2/3}<s\leq 10/9.$ Thus the proof of Lemma \ref{patch} is completed.
  \end{proof}
  
   \begin{lemma}\label{PA}
There is a constant $C$ independent of $x_\mu,$  such that
 for any $u\in  C_0^\infty\big(\mathbb{R}^{2n}\big),$ one has
  \begin{equation*} 
   \norm{ \comi{\partial_x
   V(x_\mu)}^{2/3}u}_{L^2}^2+\norm{\comi{D_x}^{2/3} u}_{L^2}^2
   \leq
   C\set{
   \norm{P_{x_\mu} u}_{L^2}^2+\norm{u}_{L^2}^2},
\end{equation*}
or equivalently,
\begin{eqnarray*}
	 \norm{ \inner{1+\abs{\partial_x
   V(x_\mu)}^2-\Delta_x}^{1/3}u}_{L^2}^2
   \leq
   C\set{
   \norm{P_{x_\mu} u}_{L^2}^2+\norm{u}_{L^2}^2},
\end{eqnarray*}
where the fractional Laplacian is defined by
$$\mathscr F_x\set{ \inner{1+\abs{\partial_x
   V(x_\mu)}^2-\Delta_x}^{1/3} u}= \inner{1+\abs{\partial_x
   V(x_\mu)}^2+\xi^2}^{1/3} \mathscr F_x u. $$
\end{lemma}

\begin{proof}
	This follows from classical hypoelliptic technique, seeing  for instance \cite[Proposition
5.22]{MR2130405}.  We omit it here for brevity.  
\end{proof}

  \begin{proof}
  	[Completeness of the proof of Theorem \ref{+Hypo}:  subelliptic estimate]  In this part we will prove the subelliptic estimate \eqref{+++A1}.  Let $\varphi_\mu,\mu\geq1,$ be the partition of unit given in Lemma \ref{uni} and let $\tau$ be given in \eqref{+++A1}.  Then  we use Lemma  \ref{lemreg} to compute
  \begin{equation*}
     \norm{\comi{D_x}^{\tau} u}_{L^2}^2
    \leq  C\sum_{\mu\geq
   1}\norm{\inner{1-\triangle_x}^{ \tau/2}\varphi_\mu
    u}_{L^2}^2+C\norm{Pu}_{L^2}^2+C\norm{u}_{L^2}^2,
  \end{equation*}
  and moreover,  observing $\tau\leq 2/3$ and  using Fourier transform in $x$ if necessary,     
  \begin{equation*}
  \begin{aligned}
 &  \norm{\inner{1-\triangle_x}^{ \tau/2}\varphi_\mu
    u}_{L^2}^2\leq     \norm{\big(1+ \abs{\partial_xV(x_\mu)}^2-\triangle_x\big)^{\tau/2}
    \varphi_\mu
    u}_{L^2}^2\\
    &=   \norm{\big(1+ \abs{\partial_xV(x_\mu)}^2-\triangle_x\big)^{1/3}
    \big(1+ \abs{\partial_xV(x_\mu)}^2-\triangle_x\big)^{\frac{\tau}{2}-\frac{1}{3}}\varphi_\mu
    u}_{L^2}^2\\
    &\leq \norm{\big(1+ \abs{\partial_xV(x_\mu)}^2-\triangle_x\big)^{1/3}
    \comi{\partial_xV(x_\mu)}^{\tau-\frac{2}{3} }\varphi_\mu
    u}_{L^2}^2\\
    &\leq C\norm{P_{x_\mu} \comi{\partial_xV(x_\mu)}^{-\tilde s}\varphi_\mu u}_{L^2}^2+C\norm{\varphi_\mu u}_{L^2}^2
  \end{aligned}
  \end{equation*}
  where  $\tilde s=\frac{2}{3}-\tau\geq 0,$  and the last inequality follows from Lemma \ref{PA}.  As a result, combining the above estimates yields 
  \begin{equation*}
     \norm{\comi{D_x}^{\tau} u}_{L^2}^2
    \leq  C\sum_{\mu\geq
   1}\norm{P_{x_\mu} \comi{\partial_xV(x_\mu)}^{-\tilde s}\varphi_\mu u}_{L^2}^2 +C\norm{Pu}_{L^2}^2+C\norm{u}_{L^2}^2.
  \end{equation*}
  Thus the desired subelliptic estimate \eqref{+++A1} will follow
    if the following 
  \begin{equation}\label{09865}
    \sum_{\mu\geq 1}\norm{P_{x_\mu}\comi{\partial_xV(x_\mu)}^{-\tilde s} \varphi_\mu\,u}_{L^2}^2
    \leq C\set{\norm{P u}_{L^2}^2
    +\norm{u}_{L^2}^2} 
  \end{equation}
  holds for all $u\in C_0^\infty(\mathbb R^{2n})$, recalling $\tilde s=\frac{2}{3}-\tau$ with $\tau$ given in \eqref{+++A1}.  To prove \eqref{09865} we write
  \[
    \comi{\partial_xV(x_\mu)}^{-\tilde s}\varphi_\mu u
    =\big(\comi{\partial_xV(x)}^{\tilde s}\comi{\partial_xV(x_\mu)}^{-\tilde s}\big)
    \varphi_\mu  \comi{\partial_xV(x)}^{-\tilde s} u.
  \]
  Then
  \begin{eqnarray*}
  	\sum_{\mu\geq 1}\norm{P_{x_\mu}  \comi{\partial_xV(x_\mu)}^{-\tilde s}\varphi_\mu u}_{L^2}^2\leq S_1+S_2,
  \end{eqnarray*}
  with
  \begin{eqnarray*}
  \begin{aligned}
    S_1&= 2 \sum_{\mu\geq 1}\norm{\big( \comi{\partial_xV(x)}^{\tilde s}\comi{\partial_xV(x_\mu)}^{-\tilde
    s}\big)
    P_{x_\mu}\varphi_\mu   \comi{\partial_xV(x)}^{-\tilde s}
    u}_{L^2}^2\\
    S_2&=2\sum_{\mu\geq 1}\norm{\big[P_{x_\mu},~ \comi{\partial_xV(x)}^{\tilde s}\comi{\partial_xV(x_\mu)}^{-\tilde s}
    \big]\varphi_\mu   \comi{\partial_xV(x)}^{-\tilde s}
    u}_{L^2}^2.
    \end{aligned}
  \end{eqnarray*}
  Using \eqref{7310} gives
  \begin{eqnarray*}
    S_1&\leq& C\sum_{\mu\geq 1}\norm{
    P_{x_\mu}\varphi_\mu   \comi{\partial_xV(x)}^{-\tilde s}
    u}_{L^2}^2\\
    &\leq & C
    \set{\norm{P
    u}_{L^2}^2+\norm{P
     \comi{\partial_xV(x)}^{-\tilde s}u}_{L^2}^2+\norm{ \comi{\partial_xV(x)}^{-\tilde
    s}u}_{L^2}^2}\\
    &\leq & C \norm{Pu}_{L^2}^2+C \norm{u}_{L^2}^2,
  \end{eqnarray*}
  where in the first inequality  we use\eqref{wh} and Lemma \ref{patch}, and
  the last inequality holds because 
  \begin{multline}\label{coma}
  	\norm{P \comi{\partial_xV(x)}^{-\tilde s}u
    }_{L^2}^2 \leq  2 \norm{Pu}_{L^2}^2+2 \norm{\big[P,~ \comi{\partial_xV(x)}^{-\tilde s}\,\big]u}_{L^2}^2\\
  \leq   C \norm{Pu}_{L^2}^2+C \norm{  \comi{\partial_xV(x)}^{-\tilde s+s-1}\comi{y}u}_{L^2}^2\leq C \norm{Pu}_{L^2}^2+C \norm{u}_{L^2}^2
  \end{multline}
  due to \eqref{elv} and \eqref{uy} as well as the fact that $-\tilde s+s-1\leq 0$ and $\tilde s\geq 0.$    
  Similarly, following the argument in \eqref{coma} we have
  \begin{eqnarray*}
  	S_2\leq C
    \sum_{\mu\geq 1}\norm{\varphi_\mu\comi{y}
    u}_{L^2}^2\leq  
 C\norm{\comi{y}u}_{L^2}^2\leq C \norm{Pu}_{L^2}^2+C \norm{u}_{L^2}^2.
  \end{eqnarray*}
  This with the estimate on $S_1$ yields  \eqref{09865}, and thus the subelliptic estimate \eqref{+++A1} follows. The proof of  Theorem \ref{+Hypo} is  completed.
  	  \end{proof}

    \section{Maximal estimate}\label{sec22}  
    
    In this part we investigate the maximal estimate, i.e., Corollary  \ref{cormax}. 
    First we list some commutation relations to be used below.  Let $Q$ and $L_j, 1\leq j\leq n,$ be given at the beginning of Section \ref{sec3}.   By direct verification we have 
\begin{equation}\label{relof}
	[L_j,~L_k]=[L_j^*,~L_k^*]=0,\quad [L_j,~L_k^*]=1 \ \textrm{if}\  j=k  \ \textrm{and}\  0 \ \textrm{otherwise},
\end{equation}
and moreover
\begin{equation}\label{ce}
	 [Q,~L_j^*]=\partial_{x_j}-\frac{1}{2}\partial_{x_j}
   V(x)+H(x)\cdot\inner{\bm e_j\wedge \partial_y}-\frac{1}{2}H(x)\cdot\inner{y\wedge \bm e_j},
\end{equation}
where  $\bm e_j=(0,\cdots, 1, \cdots,0)\in\mathbb R^n$ with only $j$-th component  equal to $1$.

\begin{proof}[Proof of Corollary \ref{cormax}]   Using
\eqref{elv} gives
\begin{eqnarray*}
  \norm{L_jL_j^*u}_{L^2}^2
  &\leq& {\rm Re}\comi{PL_j^*u,~L_j^*
  u}_{L^2}\\
  &=&{\rm Re}\comi{[P,~L_j^*]u,~L_j^*
  u}_{L^2}+ {\rm Re}\comi{Pu,~L_jL_j^*
  u}_{L^2}\\
  &\leq&{\rm Re}\comi{[P,~L_j^*]u,~L_j^*
  u}_{L^2}+{1\over2}
  \norm{L_jL_j^*u}_{L^2}^2+2\norm{Pu}_{L^2}^2.
\end{eqnarray*}
Hence
\begin{equation}\label{les}
  \norm{L_jL_j^*u}_{L^2}^2
  \leq 2\abs{\comi{[P,~L_j^*]u,~L_j^*
  u}_{L^2}}+4\norm{Pu}_{L^2}^2.
\end{equation}
Moreover it follows from \eqref{relof}-\eqref{ce} that
\[
  [P,~L_j^*]=\partial_{x_j}-\frac{1}{2}\partial_{x_j}
   V(x)+H(x)\cdot\inner{\bm e_j\wedge \partial_y}-\frac{1}{2}H(x)\cdot\inner{y\wedge \bm e_j}+L_j^*,
\]
and thus, by virtue of \eqref{eh}, 
\begin{eqnarray*}
\begin{aligned}
  & \abs{\comi{[P,~L_j^*]u,~L_j^*u}_{L^2}} \\
  &\leq 
   \comi{L_j^*u,~L_j^*u}_{L^2}
   +\Big|\big<\big(\partial_{x_j}-\frac{1}{2}\partial_{x_j}
   V +H \cdot\inner{\bm e_j\wedge \partial_y}-\frac{1}{2}H \cdot\inner{y\wedge \bm e_j}\big)u,~L_j^*u\big>_{L^2}\Big|\\
   &\leq  C\set{\norm{\comi{\partial_{x}V}^{2/3}u}_{L^2}^2
   +\norm{\comi{D_x}^{2/3}u}_{L^2}^2+\norm{L_ju}_{L^2}^2+\norm{u}_{L^2}^2}\\
   &\quad +C\set{\norm{L_j\comi{\partial_{x}V }^{1/3}u}_{L^2}^2+
   \norm{L_j\comi{D_x}^{1/3}u}_{L^2}^2}\\
   &\leq  C\set{\norm{\comi{\partial_{x}V}^{2/3}u}_{L^2}^2
   +\norm{\comi{D_x}^{2/3}u}_{L^2}^2+\norm{Pu}_{L^2}^2+\norm{u}_{L^2}^2}+
   C\norm{L_j\comi{D_x}^{1/3}u}_{L^2}^2,
   \end{aligned}
\end{eqnarray*}
the last inequality following from Lemma \ref{lewe} and \eqref{elv}.   For the last term  on the right side  we use \eqref{relof} to compute directly, for $\eps>0,$
\begin{eqnarray*}
  \norm{L_j\comi{D_{x}}^{1/3}u}_{L^2}^2= \big<L_j^*L_j u,\  \comi{D_{x}}^{2/3}u\big>_{L^2}  \leq 
  \eps\norm{L_jL_j^*u}_{L^2}^2+C_\eps\set { \norm{\comi{D_{x} }^{2\over 3}u}_{L^2}^2
  +\norm{u}_{L^2}^2}.
\end{eqnarray*}
As a result, combining the above estimates we obtain
\begin{equation*}\label{0909281}
   \abs{\comi{[P, L_j^*]u, L_j^*u}_{L^2}}
   \leq \eps \norm{L_jL_j^*u}_{L^2}^2 +
   C_\eps\Big\{\norm{\comi{\partial_{x}V}^{2\over3}u}_{L^2}^2
   +\norm{\comi{D_x}^{2\over3}u}_{L^2}^2
   +\norm{Pu}_{L^2}^2+\norm{ u}_{L^2}^2\Big\},
\end{equation*}
which, together with \eqref{les} and the   estimates \eqref{weg} and \eqref{+++A1} with $\tau=2/3$ therein,    yields
\begin{eqnarray*}
	\sum_{1\leq j\leq n} \norm{L_jL_j^*u}_{L^2}^2 &\leq & C\Big\{\norm{\comi{\partial_{x}V}^{2\over3}u}_{L^2}^2
   +\norm{\comi{D_x}^{2\over3}u}_{L^2}^2
   +\norm{Pu}_{L^2}^2+\norm{ u}_{L^2}^2\Big\} \\
   &\leq & C\Big\{ \norm{Pu}_{L^2}^2+\norm{ u}_{L^2}^2\Big\}.
\end{eqnarray*}
  The gives the assertion in Corollary \ref{cormax}, completing the proof. \end{proof}

\bigskip
\noindent{\bf Acknowledgements}  This work was supported by  NSFC (Nos. 11871054, 11961160716, 11771342)  and   Fundamental Research Funds for the Central Universities(2042020kf0210).


\end{document}